\documentclass[12pt]{article}

\usepackage{amsmath}
\usepackage{amssymb}
\usepackage{amscd}
\usepackage{amsthm}
\usepackage{indentfirst}
\usepackage[english,french]{babel}

\theoremstyle{plain}
\newtheorem*{mainthm}{Th\'eor\`eme}
\newtheorem{thm}{Th\'eor\`eme}[subsection]
\newtheorem{lem}[thm]{Lemme}
\newtheorem{cor}[thm]{Corollaire}
\newtheorem{prop}[thm]{Proposition}
\newtheorem{propdfn}[thm]{Proposition-d\'efinition}
\theoremstyle{definition}
\newtheorem{dfn}[thm]{D\'efinition}
\newtheorem{rmq}[thm]{Remarque}

\numberwithin{equation}{subsection}

\DeclareMathOperator{\Spec}{Spec}
\DeclareMathOperator{\pic}{Pic}
\DeclareMathOperator{\Div}{Div}
\DeclareMathOperator{\diviseur}{div}
\DeclareMathOperator{\Divp}{DivPrinc}
\DeclareMathOperator{\DivRat}{DivRat}
\DeclareMathOperator{\cub}{Cub}

\newcommand{\A}{\mathcal{A}}
\newcommand{\gm}{\mathbf{G}_{{\rm m}}}
\newcommand{\gmk}{\mathbf{G}_{{\rm m},K}}
\newcommand{\gmu}{\mathbf{G}_{{\rm m},U}}
\newcommand{\gmeta}{\mathbf{G}_{{\rm m},\eta}}
\newcommand{\gmlog}{\mathbf{G}_{{\rm m,log}}}
\newcommand{\jcl}{j_{{\rm cl}}}

\DeclareMathOperator{\homr}{Hom}
\DeclareMathOperator{\Hom}{\underline{Hom}}
\DeclareMathOperator{\ext}{Ext}
\DeclareMathOperator{\Ext}{\underline{Ext}}
\DeclareMathOperator{\biext}{Biext}

\begin{document}

\title{Prolongement de biextensions et accouplements en cohomologie log plate}

\author{Jean Gillibert}

\date{avril 2009}

\maketitle

\begin{abstract}
Nous revisitons, dans le langage des log sch\'emas, le probl\`eme de prolongement de biextensions de sch\'emas en groupes commutatifs lisses  par le groupe multiplicatif \'etudi\'e par Grothendieck dans \cite{gro7}. Nous montrons que ce probl\`eme admet en g\'en\'eral une solution dans la cat\'egorie des faisceaux pour la topologie log plate, contrairement \`a ce que l'on peut observer en topologie fppf pour laquelle Grothendieck a d\'efini des obstructions monodromiques. En particulier, dans le cas d'une vari\'et\'e ab\'elienne et de sa duale, il est possible de prolonger la biextension de Weil sur la totalit\'e des mod\`eles de N\'eron ; ceci permet de d\'efinir un accouplement sur les points qui combine l'accouplement de classes d\'efini par Mazur et Tate et l'accouplement de monodromie.

\medskip

\begin{otherlanguage}{english}
\begin{center}
\textbf{Abstract}
\end{center}

We study, using the language of log schemes, the problem of extending  biextensions of smooth commutative group schemes by the multiplicative group. This was first considered by Grothendieck in \cite{gro7}. We show that this problem admits a solution in the category of sheaves for Kato's log flat topology, in contradistinction to what can be observed using the fppf topology, for which monodromic obstructions were defined by Grothendieck. In particular, in the case of an abelian variety and its dual, it is possible to extend the Weil biextension to the whole N\'eron model. This allows us to define a pairing on the points which combines the class group pairing defined by Mazur and Tate and Grothendieck's monodromy pairing.
\end{otherlanguage}
\end{abstract}


\section{Introduction}

Soit $S$ un trait, de point g\'en\'erique $\eta=\Spec(K)$, et soit $A_K$ une $K$-vari\'et\'e ab\'elienne. Il est bien connu que la vari\'et\'e duale $A_K^t$ de $A_K$ est munie d'un isomorphisme canonique
\begin{equation}
\label{isocanonique}
\begin{CD}
A_K^t @>\sim >> \Ext^1_{\text{\rm \'et}}(A_K,\gmk)
\end{CD}
\end{equation}
qui est parfois utilis\'e comme d\'efinition de $A_K^t$. On peut se demander ce qu'il advient de cet isomorphisme si l'on remplace les vari\'et\'es ab\'eliennes $A_K$ et $A_K^t$ par leurs mod\`eles de N\'eron sur $S$, que nous noterons respectivement $\A$ et $\A^t$.

Grothendieck a utilis\'e le concept de biextension (introduit par Mumford) pour \'etudier ce probl\`eme. Plus pr\'ecis\'ement, prolonger l'isomorphisme \eqref{isocanonique} au niveau des mod\`eles de N\'eron \'equivaut \`a prolonger la biextension de Weil $W_K$ de $(A_K, A_K^t)$ par $\gmk$ en une biextension de $(\A,\A^t)$ par $\gm$. Soit $\Phi$ (resp. $\Phi'$) le groupe des composantes de (la fibre sp\'eciale de) $\A$ (resp. $\A^t$), Grothendieck construit dans \cite[expos\'e VIII, th\'eor\`eme 7.1, b)]{gro7} un accouplement (dit  de monodromie)
$$
\begin{CD}
\Phi\times\Phi' @>>> \mathbb{Q}/\mathbb{Z} \\
\end{CD}
$$
qui repr\'esente l'obstruction \`a l'existence d'un tel prolongement. Plus pr\'ecis\'ement, si $\Gamma$ et $\Gamma'$ sont des sous-groupes respectifs de $\Phi$ et $\Phi'$, alors $W_K$ se prolonge en une biextension de $(\A^{\Gamma}, \A^{t,\Gamma'})$ par $\gm$ si et seulement si $\Gamma$ et $\Gamma'$ sont orthogonaux sous l'accouplement (voir le paragraphe \ref{accouplements}), lequel prolongement induit alors un morphisme
$$
\begin{CD}
\gamma:\A^{t,\Gamma'} @>>> \Ext^1_{\text{\rm \'et}}(\A^{\Gamma},\gm)
\end{CD}
$$
entre faisceaux sur le site $(Lis/S)_{\text{\rm \'et}}$ des $S$-sch\'emas lisses, muni de la topologie \'etale. Si en outre $\Gamma=0$ et $\Gamma'=\Phi'$ (ou si l'accouplement de monodromie est non d\'eg\'en\'er\'e et si $\Gamma'$ est l'orthogonal de $\Gamma$), alors $\gamma$ est un isomorphisme. Ce r\'esultat est explicit\'e dans \cite[Appendix C, Prop. C.14]{milne86}. Notons que l'id\'ee de consid\'erer le site lisse appara\^it d\'ej\`a (sous forme d'une
\og{}Autocritique\fg{}) dans \cite[expos\'e VIII, 6.9]{gro7}.

Dans cet article, nous nous sommes int\'eress\'es au probl\`eme de prolongement des biextensions dans la cat\'egorie des log sch\'emas, munie de la topologie Kummer log plate (voir le paragraphe \ref{logref} pour les log r\'ef\'erences). Munissons le trait $S$ de sa log structure canonique ; en travaillant dans la cat\'egorie des faisceaux pour la topologie Kummer log plate sur $S$, nous montrons le r\'esultat suivant (cf. th\'eor\`eme \ref{sga7style}) :

\begin{mainthm}
\begin{enumerate}
\item[$(i)$] Soit $P$ un sch\'ema en groupes commutatif lisse de type fini sur $S$. Alors le morphisme de restriction \`a la fibre g\'en\'erique
$$
\ext^1_{\text{\rm kpl}}(P,\gm)\longrightarrow \ext^1_{\text{\rm pl}}(P_{\eta},\gmeta)
$$
est un isomorphisme.
\item[$(ii)$] Soient $P$ et $Q$ deux sch\'emas en groupes commutatifs lisses de type fini sur $S$. Alors le morphisme de restriction \`a la fibre g\'en\'erique
$$
\biext^1_{\text{\rm kpl}}(P,Q;\gm)\longrightarrow \biext^1_{\text{\rm pl}}(P_{\eta},Q_{\eta};\gmeta)
$$
est un isomorphisme.
\end{enumerate}
\end{mainthm}

Une fois ce r\'esultat \'etabli on retrouve, par le truchement d'une suite spectrale comparant la cohomologie log plate avec la cohomologie fppf, les obstructions d\'efinies par Grothendieck dans \cite[expos\'e VIII, th\'eor\`eme 7.1]{gro7} (cf. th\'eor\`eme \ref{retourmonodromie}). Dans le cas particulier o\`u $(P,Q)=(\A,\A^t)$, on obtient ainsi une interpr\'etation logarithmique de l'accouplement de monodromie (cf. proposition \ref{restrwlog}).

Notre principale motivation est l'existence, dans le monde logarithmique, d'une unique biextension $W^{\rm log}\in\biext^1_{\text{\rm kpl}}(\A,\A^t;\gm)$ prolongeant la biextension de Weil. Il d\'ecoule en outre du th\'eor\`eme ci-dessus que le morphisme
$$
\begin{CD}
\gamma^{\rm log}:\A^t @>>> \Ext^1_{\text{\rm kpl}}(\A,\gm)
\end{CD}
$$
induit par $W^{\rm log}$ est un isomorphisme entre faisceaux sur le site lisse de $S$.

Dans le cas global (o\`u la base $S$ n'est plus forc\'ement un trait), on peut d\'efinir \`a l'aide de $W^{\rm log}$ un accouplement de classes logarithmique (cf. d\'efinition \ref{defpairings}) qui combine l'accouplement de classes --- d\'efini par Mazur et Tate dans \cite{mt} --- et celui de monodromie (cf. proposition \ref{proppairings}).

La biextension $W^{\rm log}$  est un outil qui nous permettra, dans l'article \cite{gil5}, d'\'etudier la ramification de certains torseurs (sous des sch\'emas en groupes finis et plats) obtenus gr\^ace au cobord d'une isog\'enie entre vari\'et\'es ab\'eliennes. Plus pr\'ecis\'ement, on peut affiner les r\'esultats de \cite{gil1} et \cite{gil2} en se d\'ebarassant des conditions sur les groupes de composantes, le prix \`a payer \'etant le passage des sch\'emas aux log sch\'emas. Nous renvoyons le lecteur \`a l'introduction de \cite{gil5} pour plus de d\'etails.

R\'esumons bri\`evement le contenu de cet article. Dans la section \ref{sorites} nous introduisons les outils logarithmiques, les sites et les faisceaux dont nous aurons besoin.

Dans la section \ref{fibreslog} nous calculons le groupe $H^1_{\text{\rm kpl}}(S,\gm)$  en cohomologie Kummer log plate, $S$ \'etant un sch\'ema  n\oe{}th\'erien r\'egulier muni de la log structure d\'efinie par un diviseur \`a croisements normaux. Nous donnons une description de ce groupe analogue \`a celle du groupe de Picard usuel en termes de diviseurs, la diff\'erence \'etant que les composantes irr\'eductibles du diviseur d\'efinissant la log structure apparaissent avec des coefficients rationnels dans le groupe des \og{}diviseurs logarithmiques\fg{}. Nous en d\'eduisons, via la th\'eorie de Kummer, une description du groupe $H^1_{\text{\rm kpl}}(S,\mu_n)$.

Dans la section \ref{applications}, nous montrons le th\'eor\`eme \ref{sga7style} qui constitue le r\'esultat principal de cet article. Nous appliquons ce r\'esultat \`a la d\'efinition de l'accouplement de classes logarithmique puis, \`a l'aide de la th\'eorie des torseurs cubistes, nous relions ce nouvel accouplement \`a une construction de Moret-Bailly \cite{mb}, en retrouvant au passage l'expression, due \`a Bosch et Lorenzini \cite{bl}, de l'accouplement de monodromie en termes du symbole de N\'eron.

Pour conclure, l'auteur signale que si l'on remplace la topologie Kummer log plate par la topologie Kummer log \'etale, la plupart des  r\'esultats \'enonc\'es ici se transposent facilement, dans une forme affaiblie tenant compte de la torsion des faisceaux qui n'est pas premi\`ere aux caract\'eristiques r\'esiduelles des points de $S$ (voir la remarque \ref{logetalefailure}).


\section{Sorites faisceautiques}
\label{sorites}


\subsection{Log sch\'emas}
\label{logref}
Nous renvoyons le lecteur \`a \cite[section 1]{illusie} pour les d\'efinitions de base concernant les log sch\'emas. Toutes les log structures consid\'er\'ees ici sont fines et satur\'ees (fs).
Si $T$ est un log sch\'ema, nous noterons $M_T$ le faisceau de mono\"ides (sur le petit site \'etale de $T$) d\'efinissant la log structure de $T$. Nous noterons $M_T^{\rm gp}$ le faisceau en groupes des fractions totales du faisceau de mono\"ides $M_T$.

Nous renvoyons \`a Kato \cite{kato2} et Nizio\l{} \cite{niziol} pour les d\'efinitions et propri\'et\'es des topologies Kummer log plate et Kummer log \'etale (pour \'epargner au lecteur de trop nombreuses r\'ep\'etitions, nous omettrons tant\^ot le log, tant\^ot le Kummer). Nous travaillerons en g\'en\'eral avec la premi\`ere.

Si $X$ est un sch\'ema (resp. un log sch\'ema), nous noterons $(Sch/X)$ (resp. $(fs/X)$) la cat\'egorie des sch\'emas (resp. des log sch\'emas fs) sur $X$. Nous utiliserons les symboles \'et et pl pour d\'esigner les topologies \'etale et fppf classiques, ainsi que k\'et et kpl pour d\'esigner les topologies Kummer log \'etale et Kummer log plate.
Si $C$ est l'une de ces cat\'egories et si top est l'une de ces topologies, nous noterons $C_{\text{top}}$ le site obtenu en munissant $C$ de la topologie top.

Le foncteur (de la cat\'egorie $(fs/X)^{op}$ dans celle des groupes ab\'eliens)
$$
\gmlog:T\longmapsto \Gamma(T,M_T^{\rm gp})
$$
est un faisceau pour la topologie Kummer plate (voir \cite[Theorem 3.2]{kato2} ou \cite[Cor. 2.22]{niziol}). Rappelons au passage que $\gmlog$ (not\'e $\gm^{\rm cpt}$ dans et \cite{kato2} et $\gm^{\times}$ dans \cite{niziol}) n'est pas repr\'esentable par un log sch\'ema (voir \cite[2.7 (c)]{illusie}).

Si $X$ est un sch\'ema, nous consid\'ererons \'egalement la cat\'egorie $(Lis/X)$ des $X$-sch\'emas lisses. Cette derni\`ere, munie de la topologie \'etale, et un site que l'on appelle le site lisse (ou lisse-\'etale) de $X$, et que nous noterons $(Lis/X)_{\text{\rm \'et}}$.

Nous rappelons dans l'\'enonc\'e ci-dessous la d\'efinition de la log structure associ\'ee \`a l'ouvert compl\'ementaire d'un diviseur (\`a croisements normaux) sur un sch\'ema r\'egulier, laquelle a la vertu de commuter au changement de base lisse.

\begin{propdfn}
\label{logstrdefs}
Soient $X$ un sch\'ema n\oe{}th\'erien r\'egulier, et $j:U\rightarrow X$ un ouvert dont le compl\'ementaire est un diviseur \`a croisements normaux sur $X$.
\begin{enumerate}
\item[$(1)$] L'inclusion
$$
\mathcal{O}_X\cap j_*\mathcal{O}_U^*\longrightarrow \mathcal{O}_X
$$
induit une log structure fine et satur\'ee sur $X$. On dit qu'elle est d\'efinie par $U$, lequel est l'ouvert de trivialit\'e de cette log structure.
\item[$(2)$] La log structure d\'efinie par $U$ est l'image directe par $j$ de la log structure triviale sur $U$. En particulier, $M_X^{\rm gp}=j_*\mathcal{O}_U^*$.
\item[$(3)$] Soit $f:Y\rightarrow X$ un morphisme lisse. Alors l'image r\'eciproque par $f$ de la log structure d\'efinie par $U$ est la log structure d\'efinie par $f^{-1}(U)$.
\end{enumerate}
\end{propdfn}

\begin{proof}
Voir \cite[1.5 et 2.5]{kato1} ou \cite[1.7]{illusie} pour le point $(1)$. Le fait que la log structure ainsi d\'efinie soit l'image directe de la log structure triviale de $U$ est mentionn\'e dans \cite[1.5]{kato1} et implique (par la propri\'et\'e universelle de $M_X^{\rm gp}$) la relation $M_X^{\rm gp}=j_*M_U^{\rm gp}=j_*\mathcal{O}_U^*$. Le point $(3)$ se r\'esume au fait suivant : sous les hypoth\`eses envisag\'ees ($X$ \'etant normal et $U$ dense dans $X$), la formation du faisceau $j_*\mathcal{O}_U^*$ commute au changement de base lisse (voir \cite[expos\'e VIII, 6.8 b) et 6.9]{gro7}).
\end{proof}

\begin{rmq}
$(i)$ Si $X$ est un trait et $\eta$ son point g\'en\'erique, alors la log structure d\'efinie par $\eta$ s'appelle la log structure canonique.

$(ii)$ Dans le point $(3)$, l'hypoth\`ese de lissit\'e de $f$ est n\'ecessaire. Si l'on consid\`ere un morphisme (fini et plat) de traits $f:X'\rightarrow X$ d'indice de ramification $>1$, alors l'image r\'eciproque par $f$ de la log structure canonique de $X$ n'est pas la log structure canonique de $X'$.
\end{rmq}

Dans la suite de cet article, nous fixons les notations et hypoth\`eses suivantes : $S$ est un sch\'ema  n\oe{}th\'erien r\'egulier, et $D=\sum_{m=1}^r D_m$ est un diviseur \`a croisements normaux \`a multiplicit\'es $1$ sur $S$ (les $D_m$ \'etant des diviseurs irr\'eductibles et r\'eduits). Nous notons $j:U\rightarrow S$ l'ouvert compl\'ementaire de $D$. Dor\'enavant, nous consid\'erons $S$ comme un log sch\'ema, muni de la log structure d\'efinie par $U$.

Nous noterons $\gm$ le groupe multiplicatif sur $S$, et $\mu_n$ le sch\'ema en groupe des racines $n$-i\`emes de l'unit\'e sur $S$.

Si $G$ est un $S$-sch\'ema en groupes, nous noterons \'egalement $G$ le log sch\'ema en groupes obtenu en munissant $G$ de la log structure image r\'eciproque de celle de $S$. Tous les log sch\'emas en groupes consid\'er\'es ici sont obtenus de cette mani\`ere.


\subsection{Images directes sur le site lisse}

Tous les morphismes de sites consid\'er\'es ici donnent lieu \`a des morphismes de topos au sens de \cite[expos\'e IV]{gro4} (en particulier, le foncteur image inverse est exact, et le foncteur image directe est exact \`a gauche, et envoie les injectifs sur des injectifs).

Pour tout morphisme de log sch\'emas $T\rightarrow S$, son pull-back $T_U\rightarrow U$ par $j$ est un morphisme de sch\'emas. De plus, cette op\'eration transforme les recouvrements pour la topologie Kummer plate en recouvrements pour la topologie fppf. On obtient ainsi un morphisme de sites
$$
j:(Sch/U)_{\text{\rm pl}}\longrightarrow (fs/S)_{\text{\rm kpl}}
$$
qui est un morphisme de localisation. Nous noterons d'autre part
$$
\jcl:(Lis/U)_{\text{\rm \'et}}\longrightarrow (Lis/S)_{\text{\rm \'et}}
$$
le morphisme (classique) entre sites lisses.

Suivant les notations de Kato, nous appellerons $\varepsilon$ le morphisme de sites
$$
\begin{CD}
\varepsilon:(fs/S)_{\text{\rm kpl}} @>>> (fs/S)_{\text{\rm pl}}
\end{CD}
$$
qui consiste \`a comparer la topologie Kummer log plate et la topologie fppf classique.

Nous consid\'ererons \'egalement le morphisme de sites
$$
\begin{CD}
q:(fs/S)_{\text{\rm kpl}} @>>> (Lis/S)_{\text{\rm \'et}}
\end{CD}
$$
obtenu en composant les morphismes
$$
\begin{CD}
(fs/S)_{\text{\rm kpl}} @>\varepsilon>> (fs/S)_{\text{\rm pl}} @>m>> (Lis/S)_{\text{\rm \'et}}
\end{CD}
$$
o\`u $m$ est un morphisme de comparaison de la topologie fppf et de la topologie \'etale.

\begin{lem}
\label{fx1}
On a l'\'egalit\'e $q_*\gmlog=(\jcl)_*\gmu$ sur le site lisse de $S$.
\end{lem}

\begin{proof}
D\'ecoule de la proposition-d\'efinition \ref{logstrdefs}.
\end{proof}

Il est bien connu que le morphisme naturel $\gm\rightarrow (\jcl)_*\gmu$ est un monomorphisme sur le site lisse. Nous d\'efinissons alors le faisceau $\underline{D}$ par la suite exacte
\begin{equation}
\label{suite1}
\begin{CD}
0 @>>> \gm @>>> (\jcl)_*\gmu @>>> \underline{D} @>>> 0 \\
\end{CD}
\end{equation}
sur le site lisse de $S$. On constate aussit\^ot que la restriction $\underline{D}_{\text{\rm Zar}}$ du faisceau $\underline{D}$ au (petit) site Zariski de $X$ est donn\'ee par
$$
\underline{D}_{\text{\rm Zar}}=\bigoplus_{m=1}^r \mathbb{Z}_{D_m}
$$
o\`u $D_1,\dots,D_r$ sont les composantes irr\'eductibles de $D$. Si ces derni\`eres sont en outre g\'eom\'etriquement unibranches, alors une telle description de $\underline{D}$ est valable sur le site lisse tout entier, d'apr\`es \cite[expos\'e VIII, 6.9]{gro7}.

\begin{thm}
\label{comparaison}
On dispose d'un isomorphisme canonique
$$
m_*(R^1\varepsilon_* \gm)\simeq \underline{D}\otimes_{\mathbb{Z}} \mathbb{Q}/\mathbb{Z}
$$
entre faisceaux sur le site lisse de $S$.
\end{thm}

\begin{proof}
D'apr\`es \cite[Theorem 4.1]{kato2} ou \cite[Theorem 3.12]{niziol}, nous avons un isomorphisme canonique
$$
R^1\varepsilon_*\gm \simeq \lim_{\longrightarrow} \Hom(\mu_n,\gm)\otimes_{\mathbb{Z}} (\gmlog/\gm)
$$
o\`u le faisceau quotient $(\gmlog/\gm)$ est calcul\'e pour la topologie fppf (ou la topologie \'etale, ce qui revient ici au m\^eme). En vertu du lemme \ref{fx1} et de la suite \eqref{suite1}, la restriction de ce quotient au site lisse n'est autre que $\underline{D}$. Par ailleurs, la limite inductive des $\Hom(\mu_n,\gm)$ est le faisceau constant $\mathbb{Q}/\mathbb{Z}$, d'o\`u le r\'esultat.
\end{proof}

La proposition qui suit d\'ecoule de \cite[Appendice]{gb3}.

\begin{prop}
Soit $G$ un $S$-sch\'ema en groupes lisse. Alors $R^1q_* G=m_*(R^1\varepsilon_* G)$.
\end{prop}

\begin{proof}
D'apr\`es \cite[Appendice, th\'eor\`eme 11.7]{gb3}, sous les hypoth\`eses envisag\'ees, $R^im_*G=0$ pour tout $i>0$. En effet, on peut remplacer le gros site $(fs/S)_{\text{\rm pl}}$ par le petit site $(Sch/S)_{\text{\rm pl}}$ et on se retrouve dans une situation de comparaison de la cohomologie fppf et de la cohomologie \'etale. Il ne reste plus qu'\`a d\'eriver le foncteur compos\'e $q_*=m_*\circ\varepsilon_*$ (en remarquant au passage que $q_*G=G$).
\end{proof}

\begin{cor}
\label{fx4}
On dispose d'un isomorphisme canonique
$$
R^1q_* \gm\simeq \underline{D}\otimes_{\mathbb{Z}} \mathbb{Q}/\mathbb{Z}
$$
entre faisceaux sur le site lisse de $S$.
\end{cor}

\begin{lem}
\label{fx5}
Avec les notations pr\'ec\'edentes, $R^1q_*\gmlog=0$.
\end{lem}

\begin{proof}
On peut traduire l'\'egalit\'e $R^1q_*\gmlog=0$ de la fa\c{c}on suivante : pour tout localis\'e strict $\overline{S}$ de $S$, on a
$$
H^1_{\text{\rm kpl}}(\overline{S},\gmlog)=0
$$
Ce r\'esultat a \'et\'e montr\'e par Kato (voir \cite[Corollary 5.2]{kato2} ou \cite[Corollary 3.21]{niziol}).
\end{proof}

\begin{prop}
\label{h1gmlog}
Dans le diagramme commutatif
$$
\begin{CD}
H^1_{\text{\rm \'et}}(S,(\jcl)_*\gmu) @>q^*>> H^1_{\text{\rm kpl}}(S,\gmlog) \\
@V\jcl^*VV @VVj^*V \\
H^1_{\text{\rm \'et}}(U,\gmu) @>>> H^1_{\text{\rm pl}}(U,\gmu) \\
\end{CD}
$$
toutes les fl\`eches sont des isomorphismes.
\end{prop}

\begin{proof}
La fl\`eche $q^*$ est un isomorphisme car $R^1q_*\gmlog=0$ d'apr\`es le lemme \ref{fx5}. La fl\`eche $\jcl^*$ est un isomorphisme car, le sch\'ema $S$ \'etant r\'egulier, ses anneaux locaux sont factoriels et, par un raisonnement standard, $R^1(\jcl)_*\gmu=0$ en cohomologie \'etale. Enfin, la fl\`eche horizontale du bas est un isomorphisme d'apr\`es \cite[Appendice, th\'eor\`eme 11.7]{gb3}. On en d\'eduit que $j^*$ est un isomorphisme.
\end{proof}


\subsection{Un faisceau quotient}

D'apr\`es \cite[Proposition 4.2]{kato2} on dispose, pour tout entier $n\geq 1$, d'une suite exacte pour la topologie Kummer plate (dite suite exacte de Kummer)
\begin{equation}
\label{kummerlog}
\begin{CD}
0 @>>> \mu_n @>>> \gmlog @>>> \gmlog @>>> 0. \\
\end{CD}
\end{equation}
Dans le cas o\`u $n$ est inversible sur $S$, cette suite est exacte pour la topologie Kummer \'etale (voir \cite[2.7 (d)]{illusie}).

Nous noterons $(\gmlog/\gm)^{\text{\rm kpl}}$ le quotient de $\gmlog$ par $\gm$ dans la cat\'egorie des faisceaux ab\'eliens sur $(fs/S)_{\text{\rm kpl}}$.

\begin{lem}
\label{unidiv}
Pour tout entier $n\geq 1$, la multiplication par $n$ est un automorphisme du faisceau $(\gmlog/\gm)^{\rm kpl}$.
\end{lem}

\begin{proof}
On compare les suites exactes de Kummer (en cohomologie Kummer log plate) donn\'ees par la multiplication par $n$ sur les faisceaux $\gm$ et $\gmlog$ (pour l'exactitude de ces suites, on consultera \cite[Proposition 4.2]{kato2}). Comme les deux noyaux sont \'egaux (au groupe $\mu_n$), on en d\'eduit le r\'esultat par le lemme des neuf.
\end{proof}

Soit la suite exacte sur le site Kummer plat
\begin{equation}
\label{suite2}
\begin{CD}
0 @>>> \gm @>>> \gmlog @>>> (\gmlog/\gm)^{\text{\rm kpl}} @>>> 0 \\
\end{CD}
\end{equation}
en prenant l'image directe par $q$ on obtient sur le site lisse la suite
$$
0 \longrightarrow \gm \longrightarrow j_*\gmu \longrightarrow q_*(\gmlog/\gm)^{\text{\rm kpl}} \longrightarrow R^1q_*\gm \longrightarrow R^1q_*\gmlog
$$
dans laquelle $R^1q_*\gmlog=0$ d'apr\`es le lemme \ref{fx5}. D'autre part, le quotient $\gm/j_*\gmu$ sur le site lisse est  $\underline{D}$. On en d\'eduit une suite exacte courte
$$
\begin{CD}
0 @>>> \underline{D} @>>> q_*(\gmlog/\gm)^{\text{\rm kpl}} @>>> R^1q_*\gm @>>> 0 \\
\end{CD}
$$
que nous allons expliciter un peu plus.

\begin{lem}
\label{fx7}
Avec les notations pr\'ec\'edentes, nous avons
$$
q_*(\gmlog/\gm)^{\text{\rm kpl}}=\underline{D}\otimes_{\mathbb{Z}} \mathbb{Q}
$$
sur le site lisse de $S$.
\end{lem}

\begin{proof}
D'apr\`es le lemme \ref{unidiv}, le faisceau $(\gmlog/\gm)^{\text{\rm kpl}}$ est uniquement divisible. Il en est donc de m\^eme de son image par $q_*$. Mais cette derni\`ere contient $\underline{D}$, donc contient forc\'ement $\underline{D}\otimes_{\mathbb{Z}} \mathbb{Q}$ (qui est en quelque sorte l'enveloppe uniquement divisible de $\underline{D}$). Plus exactement, l'injection canonique $\underline{D} \rightarrow q_*(\gmlog/\gm)^{\text{\rm kpl}}$ se factorise
$$
\begin{CD}
\underline{D} @>>> \underline{D}\otimes_{\mathbb{Z}} \mathbb{Q} @>h>>  q_*(\gmlog/\gm)^{\text{\rm kpl}} \\
\end{CD}
$$
On en d\'eduit un monomorphisme canonique de faisceaux
$$
\underline{D}\otimes_{\mathbb{Z}} \mathbb{Q}/\mathbb{Z}\longrightarrow R^1q_*\gm
$$
qui est un isomorphisme d'apr\`es le corollaire \ref{fx4}. On en d\'eduit ais\'ement que $h$ est un isomorphisme.
\end{proof}

\begin{rmq}
\label{logetalefailure}
On peut aussi consid\'erer le quotient $(\gmlog/\gm)^{\text{\rm k\'et}}$ dans la cat\'egorie des faisceaux ab\'eliens sur $(fs/S)_{\text{\rm k\'et}}$. La multiplication par $n$ est un automorphisme de ce faisceau si et seulement si $n$ est inversible sur $S$. On peut alors transposer en cohomologie log \'etale les r\'esultats des paragraphes 2 et 3 en rempla\c{c}ant $\mathbb{Q}$ par le sous-anneau de $\mathbb{Q}$ engendr\'e par les inverses des entiers inversibles sur $S$.
\end{rmq}


\section{Fibr\'es en droites logarithmiques}
\label{fibreslog}


\subsection{Diviseurs \`a coefficients rationnels}

Le sch\'ema $S$ \'etant r\'egulier, le groupe $\Div(S)$ des diviseurs de Cartier sur $S$ est isomorphe au groupe  $\mathcal{Z}^1(S)$ des cycles de codimension $1$ de $S$ (voir \cite[21.6.9]{ega4}). Nous comettrons un abus de langage en identifiant ces deux groupes. D'autre part, nous noterons $\Divp(S)$ le sous-groupe des diviseurs principaux (que nous identifierons au sous-groupe des cycles principaux). Sous les hypoth\`eses envisag\'ees, il est bien connu que le groupe $\pic(S)=H^1_{\text{\rm pl}}(S,\gm)$ est isomorphe \`a $\Div(S)/\Divp(S)$. Notre but est de donner ici une description analogue du groupe $H^1_{\text{\rm kpl}}(S,\gm)$.

\begin{dfn}
On appelle groupe des diviseurs \`a coefficients rationnels au-dessus de $D$, et l'on note $\DivRat(S,D)$, le sous-groupe de $\Div(S)\otimes_{\mathbb{Z}} \mathbb{Q}$ constitu\'e des diviseurs dont la restriction \`a $U$ est \`a coefficients entiers.
\end{dfn}

\begin{rmq}
Il semblerait l\'egitime d'appeler $\DivRat(S,D)$ le groupe des diviseurs logarithmiques sur $S$. On v\'erifie ais\'ement que
$$
\DivRat(S,D)=\Div(U)\oplus\bigoplus_{m=1}^r \mathbb{Q}.D_m
$$
\end{rmq}

\begin{thm}
\label{logpic}
On dispose d'un isomorphisme canonique
$$
\begin{CD}
\DivRat(S,D)/\Divp(S) @>\sim>> H^1_{\text{\rm kpl}}(S,\gm) \\
\end{CD}
$$
\end{thm}

\begin{proof}
En vertu du lemme \ref{fx7}, on peut \'ecrire
$$
\Gamma(S,(\gmlog/\gm)^{\text{\rm kpl}})=\bigoplus_{m=1}^r \mathbb{Q}.D_m
$$
et d'autre part on sait que $\Gamma(S,\gmlog)=\Gamma(U,\gm)$. Ainsi, la suite exacte \eqref{suite2} de faisceaux pour la topologie Kummer log plate donne lieu \`a la suite exacte du bas dans le diagramme commutatif ci-dessous, la suite du haut provenant quant \`a elle de la suite \eqref{suite1} de faisceaux \'etales sur le site lisse.
$$
\begin{CD}
\Gamma(U,\gm) @>>> \bigoplus_{m=1}^r \mathbb{Z}.D_m @>>> \pic(S) @>>> \pic(U) @. \rightarrow 0 \\
@| @VVV @VVV @VVV \\
\Gamma(U,\gm) @>>> \bigoplus_{m=1}^r \mathbb{Q}.D_m @>>> H^1_{\text{\rm kpl}}(S,\gm) @>>> H^1_{\text{\rm kpl}}(S,\gmlog) @. \\
\end{CD}
$$
On sait enfin (proposition \ref{h1gmlog}) que la fl\`eche verticale de droite est un isomorphisme. Par cons\'equent, la fl\`eche en bas \`a droite est \'egalement surjective. On en d\'eduit que le carr\'e du milieu est cocart\'esien, c'est-\`a-dire que $H^1_{\text{\rm kpl}}(S,\gm)$ est isomorphe \`a la somme amalgam\'ee des deux fl\`eches provenant de $\oplus_{m=1}^r \mathbb{Z}.D_m$.

D'autre part on v\'erifie ais\'ement que le carr\'e
$$
\begin{CD}
\bigoplus_{m=1}^r \mathbb{Z}.D_m @>>> \Div(S)/\Divp(S) \\
@VVV @VVV \\
\bigoplus_{m=1}^r \mathbb{Q}.D_m @>>> \DivRat(S,D)/\Divp(S) \\
\end{CD}
$$
est cocart\'esien. Comme la partie sup\'erieure gauche de ce carr\'e est isomorphe \`a celle du carr\'e pr\'ec\'edent, on en d\'eduit que les deux carr\'es sont canoniquement isomorphes, d'o\`u le r\'esultat.
\end{proof}

\begin{cor}
\label{spectralpic}
On dispose d'une suite exacte
$$
\begin{CD}
0 @>>> \pic(S) @>>> H^1_{\text{\rm kpl}}(S,\gm) @>\nu>> \bigoplus_{m=1}^r (\mathbb{Q}/\mathbb{Z}).D_m @>>> 0 \\
\end{CD}
$$
D'autre part, le morphisme naturel
$$
\xi:H^2_{\text{\rm pl}}(S,\gm) \longrightarrow H^2_{\text{\rm kpl}}(S,\gm)
$$
est injectif.
\end{cor}

\begin{proof}
On dispose d'une suite exacte (d\'eduite de suite spectrale)
$$
\begin{CD}
0 \longrightarrow \pic(S) \longrightarrow H^1_{\text{\rm kpl}}(S,\gm) @>\nu>> (\underline{D}\otimes_{\mathbb{Z}} (\mathbb{Q}/\mathbb{Z}))(S) \longrightarrow  \ker(\xi) \longrightarrow 0\\
\end{CD}
$$
D'apr\`es le th\'eor\`eme \ref{logpic}, la fl\`eche $\nu$ s'identifie \`a la fl\`eche naturelle
$$
\begin{CD}
\DivRat(S,D)/\Divp(S) @>>> \bigoplus_{m=1}^r (\mathbb{Q}/\mathbb{Z}).D_m \\
\end{CD}
$$
d\'eduite par passage au quotient de la projection naturelle
$$
\DivRat(S,D)\longrightarrow \bigoplus_{m=1}^r \mathbb{Q}.D_m
$$
On en d\'eduit que $\nu$ est surjective. Par cons\'equent, $\ker(\xi)=0$.
\end{proof}

\begin{rmq}
Si $N\in\DivRat(S,D)$ est un diviseur \`a coefficients rationnels au-dessus de $D$, nous noterons $[N]$ sa classe modulo $\Divp(S)$. Remarquons au passage un ph\'enom\`ene qui pourrait pr\^eter \`a confusion : supposons que $D_{\alpha}$ soit une composante irr\'eductible de $D$ telle que $[D_{\alpha}]=0$, alors pour tout $n\in\mathbb{N}^*$, la classe $[\frac{1}{n}D_{\alpha}]$ est un \'el\'ement d'ordre exactement $n$ dans $\DivRat(S,D)/\Divp(S)$. En d'autres termes, le symbole $[-]$ n'est pas $\mathbb{Q}$-lin\'eaire (il est en revanche $\mathbb{Z}$-lin\'eaire).
\end{rmq}


\subsection{Th\'eorie de Kummer}

Si l'on restreint \`a $U$ la suite de Kummer (en cohomologie log plate) pour le faisceau $\gmlog$ (\`a l'aide du foncteur de localisation $j^*$), on trouve la suite exacte de Kummer usuelle pour la topologie fppf sur $U$.

Soit $n\geq 1$ un entier. Par passage \`a la cohomologie dans la suite de Kummer log plate \eqref{kummerlog} (resp. dans la suite de Kummer classique sur $U$) on obtient la premi\`ere (resp. deuxi\`eme) ligne du diagramme suivant
\begin{equation}
\label{compsuitesKummer}
\begin{CD}
0 @>>> \gmlog(S)/n @>>> H^1_{\text{\rm kpl}}(S,\mu_n) @>>> H^1_{\text{\rm kpl}}(S, \gmlog)[n] @>>> 0 \\
@. @VVV @VVV @VVV \\
0 @>>> \gm(U)/n @>>> H^1_{\text{\rm pl}}(U,\mu_n)@>>> \pic(U)[n] @>>> 0 \\
\end{CD}
\end{equation}
dans laquelle les fl\`eches verticales sont obtenues par restriction \`a $U$.

\begin{prop}
\label{munlog}
Pour tout entier $n\geq 1$, le morphisme de restriction
$$
j^*:H^1_{\text{\rm kpl}}(S,\mu_n)\longrightarrow H^1_{\text{\rm pl}}(U,\mu_n)
$$
est un isomorphisme.
\end{prop}

\begin{proof}
Examinons le diagramme \eqref{compsuitesKummer}. Par d\'efinition de $\gmlog$, la fl\`eche verticale de gauche est un isomorphisme. D'apr\`es la proposition \ref{h1gmlog}, la fl\`eche verticale de droite est \'egalement un isomorphisme. On en d\'eduit que la fl\`eche verticale du milieu est un isomorphisme.
\end{proof}

\begin{prop}
\label{spectralkummer}
Pour tout entier $n\geq 1$, on dispose d'une suite exacte
$$
\begin{CD}
0 \rightarrow H^1_{\text{\rm pl}}(S,\mu_n)\rightarrow H^1_{\text{\rm kpl}}(S,\mu_n) @>\nu_n>> \bigoplus_{m=1}^r (\frac{1}{n}\mathbb{Z}/\mathbb{Z}).D_m
@>\theta_n>> \pic(S)/n \rightarrow \pic(U)/n \\
\end{CD}
$$
o\`u la fl\`eche $\theta_n$ est l'application naturelle qui, pour tout $m$, envoie $\frac{1}{n}.D_m$ sur la classe de $[D_m]$ modulo $n\pic(S)$.
\end{prop}

\begin{proof}
Le groupe $\mu_n$ \'etant fini et plat sur $S$, nous avons d'apr\`es Kato (voir \cite[Theorem 4.1]{kato2} ou \cite[Theorem 3.12]{niziol}) un isomorphisme canonique
$$
m_*(R^1\varepsilon_* \mu_n) \simeq \underline{D}\otimes_{\mathbb{Z}} \frac{1}{n}\mathbb{Z}/\mathbb{Z}
$$
semblable \`a celui du th\'eor\`eme \ref{comparaison}. La suite exacte de l'\'enonc\'e est une \'emanation de la suite spectrale
$$
0 \longrightarrow H^1_{\text{\rm pl}}(S,\mu_n)\longrightarrow H^1_{\text{\rm kpl}}(S,\mu_n)\longrightarrow (\underline{D}\otimes_{\mathbb{Z}} \frac{1}{n}\mathbb{Z}/\mathbb{Z})(S) \longrightarrow \ker(\xi_n)
$$
o\`u
$$
\xi_n:H^2_{\text{\rm pl}}(S,\mu_n) \longrightarrow H^2_{\text{\rm kpl}}(S,\mu_n)
$$
est le morphisme naturel. Consid\'erons le diagramme commutatif, \`a lignes exactes, dans lequel la premi\`ere (resp. deuxi\`eme) ligne provient de la suite exacte de Kummer pour $\gm$ en cohomologie fppf (resp. Kummer log plate).
$$
\begin{CD}
0 @>>> H^1_{\text{\rm pl}}(S,\gm)/n @>>> H^2_{\text{\rm pl}}(S,\mu_n) @>>> H^2_{\text{\rm pl}}(S, \gm)[n] @>>> 0 \\
@. @V\rho_n VV @V\xi_n VV @V\xi VV \\
0 @>>> H^1_{\text{\rm kpl}}(S,\gm)/n @>>> H^2_{\text{\rm kpl}}(S,\mu_n) @>>> H^2_{\text{\rm kpl}}(S, \gm)[n] @>>> 0 \\
\end{CD}
$$
D'apr\`es le corollaire \ref{spectralpic}, on sait que $\xi$ est injectif. On en d\'eduit que $\ker(\xi_n)=\ker(\rho_n)$, o\`u $\rho_n$ est le morphisme induit par le morphisme naturel. En vertu du th\'eor\`eme \ref{logpic}, $\rho_n$ s'identifie au morphisme naturel
$$
\rho'_n:(\Div(S)/\Divp(S))/n\longrightarrow (\DivRat(S,D)/\Divp(S))/n
$$
qui se r\'e\'ecrit
$$
\rho'_n:(\Div(S)/n)/H_n\longrightarrow (\DivRat(S,D)/n)/H_n
$$
o\`u $H_n$ est l'image de $\Divp(S)$ dans $\Div(S)/n$. D'autre part, nous avons
$$
n\DivRat(S,D)=n\Div(U)\oplus\bigoplus_{m=1}^r \mathbb{Q}.D_m
$$
d'o\`u
$$
\DivRat(S,D)/n=\Div(U)/n
$$
et
$$
(\DivRat(S,D)/n)/H_n=\pic(U)/n
$$
Au final, $\rho'_n$ est l'application naturelle
$$
\pic(S)/n\longrightarrow \pic(U)/n
$$
ce qui permet de conclure.
\end{proof}

\begin{rmq}
\label{nthpower}
Soit $\Omega$ un diviseur \`a coefficients dans $\frac{1}{n}\mathbb{Z}/\mathbb{Z}$, que nous noterons
$$
\Omega=\sum_{m=1}^r \frac{\omega_m}{n}.D_m
$$
les $\omega_m$ \'etant des \'el\'ements de $\mathbb{Z}/n\mathbb{Z}$. Si les $k_m$ sont des repr\'esentants entiers des $\omega_m$, nous dirons que le diviseur
$$
M:=\sum_{m=1}^r k_m.D_m
$$
est un $n$-rel\`evement de $\Omega$. Avec les notations de la proposition \ref{spectralkummer}, $\theta_n(\Omega)$ est alors \'egal \`a la classe de $[M]$ modulo $n\pic(S)$.

D'apr\`es la proposition \ref{spectralkummer}, \'etant donn\'e un $\mu_n$-torseur log plat $T\rightarrow S$, on peut lui associer un diviseur $\nu_n(T)$ \`a coefficients dans $\frac{1}{n}\mathbb{Z}/\mathbb{Z}$. En outre, $\theta_n(\nu_n(T))=0$ ce qui signifie que tout $n$-rel\`evement de $\nu_n(T)$ est une puissance $n$-i\`eme dans $\pic(S)$. Ceci nous pr\^ete \`a croire, de fa\c{c}on un peu imag\'ee, que le torseur $T$ est obtenu par extraction d'une racine $n$-i\`eme de ce rel\`evement. Nous pr\'ecisons ce point de vue dans \cite{gil5} en montrant que les $\mu_n$-torseurs log plats correspondent aux rev\^etements cycliques uniformes au sens de \cite{av}.
\end{rmq}


\section{Applications aux vari\'et\'es ab\'eliennes}
\label{applications}


\subsection{Prolongement d'extensions et de biextensions}

Dans ce paragraphe, nous consid\'erons le cas particulier suivant : $S$ est un trait, de point ferm\'e $s=\Spec(k)$ et de point g\'en\'erique $\eta=\Spec(K)$. Nous noterons, comme d'habitude,  $i:s\rightarrow S$ et $j:\eta\rightarrow S$ les immersions canoniques. Nous prenons $D$ \'egal au point ferm\'e de $S$, de sorte que $S$ se retrouve muni de sa log structure canonique.

Dans ce contexte, la suite \eqref{suite1} se r\'ecrit
\begin{equation}
\label{suite1bis}
\begin{CD}
0 @>>> \gm @>>> (\jcl)_*\gmeta @>\lambda>> i_*\mathbb{Z}_k @>>> 0 \\
\end{CD}
\end{equation}
la fl\`eche $\lambda$ \'etant induite par la valuation.

Le lecteur attentif remarquera le parall\`ele volontaire entre l'\'enonc\'e ci-dessous et celui de \cite[expos\'e VIII, th\'eor\`eme 7.1]{gro7}.

\begin{thm}
\label{sga7style}
\begin{enumerate}
\item[$(i)$] Soit $P$ un sch\'ema en groupes commutatif lisse sur $S$. Supposons que le groupe des composantes de la fibre sp\'eciale de $P$,
$$
\Phi:=P_s/P_s^0
$$
soit un groupe de torsion. Alors le morphisme de restriction \`a la fibre g\'en\'erique
$$
j^*:\ext^1_{\text{\rm kpl}}(P,\gm)\longrightarrow \ext^1_{\text{\rm pl}}(P_{\eta},\gmeta)
$$
est injectif. Si de plus il existe un entier $m>0$ tel que $m\Phi=0$ (ce qui est le cas en particulier si $P$ est de type fini), alors $j^*$ est un isomorphisme.
\item[$(ii)$] Soient $P$ et $Q$ deux sch\'emas en groupes commutatifs lisses sur $S$, posons
$$
\Phi:=P_s/P_s^0, \qquad \Psi:=Q_s/Q_s^0
$$
et supposons que $\Phi$ ou $\Psi$ soit un groupe de torsion. Alors le morphisme de restriction \`a la fibre g\'en\'erique
$$
j^*:\biext^1_{\text{\rm kpl}}(P,Q;\gm)\longrightarrow \biext^1_{\text{\rm pl}}(P_{\eta},Q_{\eta};\gmeta)
$$
est injectif. Si de plus il existe un entier $m>0$ tel que $m\Phi=0$ ou $m\Psi=0$, alors $j^*$ est un isomorphisme.
\end{enumerate}
\end{thm}

Comme $P$ est un objet du site $(Lis/S)_{\text{\rm \'et}}$, nous avons $q^*P=P$, d'o\`u un isomorphisme canonique
\begin{equation}
\label{isofonct}
\homr(P,-) \simeq \homr(P,q_*-)
\end{equation}
o\`u le membre de droite est calcul\'e dans la cat\'egorie des faisceaux sur $(Lis/S)_{\text{\rm \'et}}$. Ceci nous permet d'\'enoncer un premier lemme.

\begin{lem}
\label{hom01}
Avec les notations pr\'ec\'edentes, la relation
$$
\homr(P,(\gmlog/\gm)^{\text{\rm kpl}})=0
$$
est v\'erifi\'ee si $\Phi$ est de torsion. Dans ce cas, le morphisme naturel
$$
\ext^1_{\text{\rm kpl}}(P,\gm)\longrightarrow \ext^1_{\text{\rm kpl}}(P,\gmlog)
$$
est injectif.
\end{lem}

\begin{proof}
Par la relation d'adjonction \eqref{isofonct}, il suffit de montrer que
$$
\homr(P,q_*(\gmlog/\gm)^{\text{\rm kpl}})=0
$$
D'apr\`es le lemme \ref{fx7}, $q_*(\gmlog/\gm)^{\text{\rm kpl}}=\underline{D}\otimes_{\mathbb{Z}} \mathbb{Q}=i_*\mathbb{Q}_k$ (lequel est un faisceau constant tronqu\'e sans torsion). De nouveau par adjonction, il vient
$$
\homr(P,i_*\mathbb{Q}_k)=\homr(P_s,\mathbb{Q}_k)
$$
Comme $P_s^0$ est connexe, on sait que $\homr(P_s^0,\mathbb{Q}_k)=0$, d'o\`u
$$
\homr(P_s,\mathbb{Q}_k)=\homr(\Phi,\mathbb{Q}_k)
$$
Ce dernier est nul si $\Phi$ est de torsion, d'o\`u le r\'esultat. Le second point d\'ecoule du premier en appliquant le foncteur $\homr(P,-)$ \`a la suite exacte \eqref{suite2}.
\end{proof}

\begin{lem}
\label{ext1gmlog}
Dans le diagramme commutatif
$$
\begin{CD}
\ext^1_{\text{\rm \'et}}(P,(\jcl)_*\gmeta) @>q^*>> \ext^1_{\text{\rm kpl}}(P,\gmlog) \\
@V\jcl^*VV @VVj^*V \\
\ext^1_{\text{\rm \'et}}(P_{\eta},\gmeta) @>>> \ext^1_{\text{\rm pl}}(P_{\eta},\gmeta) \\
\end{CD}
$$
toutes les fl\`eches sont des isomorphismes.
\end{lem}

\begin{proof}
Le principe \'etant le m\^eme que dans la preuve de la proposition \ref{h1gmlog}, nous nous contenterons de montrer que $q^*$ est un isomorphisme. Le point de d\'epart est l'isomorphisme canonique de foncteurs \eqref{isofonct}.
Le premier foncteur d\'eriv\'e de $\homr(P,-)$ \'evalu\'e sur $\gmlog$ nous donne $\ext^1_{\text{\rm kpl}}(P,\gmlog)$. D'autre part, comme $R^1q_*\gmlog=0$ (lemme \ref{fx5}), la suite spectrale de d\'erivation des foncteurs compos\'es montre qu'en d\'erivant le foncteur  $\homr(P,q_*-)$ et en l'\'evaluant sur $\gmlog$ on obtient $\ext^1_{\text{\rm \'et}}(P,q_*\gmlog)$. D'o\`u le r\'esultat, sachant que $q_*\gmlog=(\jcl)_*\gmeta$ (lemme \ref{fx1}).
\end{proof}

Soit $P^0$ le $S$-sch\'ema obtenu en recollant la fibre g\'en\'erique de $P$ et la composante neutre de sa fibre sp\'eciale. Ainsi, $P^0$ est un sous-sch\'ema en groupes ouvert de $P$, et nous avons une suite exacte  (pour la topologie \'etale) de $S$-sch\'emas en groupes
\begin{equation}
\label{componentgroup}
\begin{CD}
0 @>>> P^0 @>>> P @>>> i_*\Phi @>>> 0 \\
\end{CD}
\end{equation}

\begin{lem}
\label{ext2}
La fl\`eche de restriction \`a la fibre g\'en\'erique
$$
(j^*)^0:\ext^1_{\text{\rm kpl}}(P^0,\gm)\longrightarrow \ext^1_{\text{\rm pl}}(P_{\eta},\gmeta)
$$
est un isomorphisme.
\end{lem}

\begin{proof}
Consid\'erons le diagramme commutatif suivant
$$
\begin{CD}
\ext^1_{\text{\rm \'et}}(P^0,\gm) @>q^*>> \ext^1_{\text{\rm kpl}}(P^0,\gm) \\
@VVV @VV(j^*)^0V \\
\ext^1_{\text{\rm \'et}}(P^0,(\jcl)_*\gmeta) @>\jcl^*>> \ext^1_{\text{\rm pl}}(P_{\eta},\gmeta) \\
\end{CD}
$$
La fl\`eche $q^*$ est (clairement) injective, et $\jcl^*$ est un isomorphisme. De plus, la fl\`eche $(j^*)^0$ s'\'ecrit comme la compos\'ee des fl\`eches
$$
\ext^1_{\text{\rm kpl}}(P^0,\gm)\longrightarrow \ext^1_{\text{\rm kpl}}(P^0,\gmlog)\longrightarrow \ext^1_{\text{\rm pl}}(P_{\eta},\gmeta)
$$
lesquelles sont toutes les deux injectives : la premi\`ere par le lemme \ref{hom01}, la deuxi\`eme par le lemme \ref{ext1gmlog}. Donc $(j^*)^0$ est injective. Enfin, la fl\`eche verticale de gauche est bijective. Nous avons en effet une suite exacte (obtenue en appliquant le foncteur $\homr(P^0,-)$ \`a la suite \eqref{suite1bis} sur le site lisse)
$$
\homr(P^0,i_*\mathbb{Z}_k)\longrightarrow \ext^1_{\text{\rm \'et}}(P^0,\gm)\longrightarrow \ext^1_{\text{\rm \'et}}(P^0,(\jcl)_*\gmeta)\longrightarrow \ext^1_{\text{\rm \'et}}(P^0,i_*\mathbb{Z}_k)
$$
dont les termes extr\'emaux sont nuls : celui de gauche par connexit\'e de $P_s^0$, celui de droite en vertu de \cite[expos\'e VIII, proposition 5.5 (i)]{gro7}. On d\'eduit de tout cela que $(j^*)^0$ et $q^*$ sont des isomorphismes.
\end{proof}

\begin{proof}[D\'emonstration du th\'eor\`eme \ref{sga7style}]
Montrons le point $(i)$.
En appliquant successivement les foncteurs $\homr(P,-)$ et $\homr(P^0,-)$ \`a la suite exacte \eqref{suite2} on obtient un diagramme commutatif, \`a lignes exactes
$$
\begin{CD}
0\rightarrow \; @. \ext^1_{\text{\rm kpl}}(P,\gm) @>>> \ext^1_{\text{\rm kpl}}(P,\gmlog)@>>> \ext^1_{\text{\rm kpl}}(P,(\gmlog/\gm)^{\text{\rm kpl}}) \\
@. @Vf_0 VV @Vf_1 VV @Vf_2 VV \\
0\rightarrow \; @. \ext^1_{\text{\rm kpl}}(P^0,\gm) @>>> \ext^1_{\text{\rm kpl}}(P^0,\gmlog)@>>> \ext^1_{\text{\rm kpl}}(P^0,(\gmlog/\gm)^{\text{\rm kpl}}) \\
\end{CD}
$$
les z\'eros \`a gauche \'etant d\'eduits du lemme \ref{hom01}. La fl\`eche $j^*$ de l'\'enonc\'e est \'egale \`a $(j^*)^0\circ f_0$ et on sait que $(j^*)^0$ est un isomorphisme d'apr\`es le lemme \ref{ext2}. Par cons\'equent, $j^*$ est injective (resp. est un isomorphisme) si et seulement si $f_0$ l'est.

Montrons tout d'abord que la fl\`eche $f_0$ est injective. En composant les fl\`eches
$$
\begin{CD}
\ext^1_{\text{\rm kpl}}(P,\gmlog)@>f_1>> \ext^1_{\text{\rm kpl}}(P^0,\gmlog)@>>> \ext^1_{\text{\rm pl}}(P_{\eta},\gmeta)
\end{CD}
$$
on obtient le morphisme naturel de restriction \`a la fibre g\'en\'erique, qui est un isomorphisme d'apr\`es le lemme \ref{ext1gmlog}. Mais la deuxi\`eme fl\`eche est \'egalement un isomorphisme d'apr\`es le m\^eme lemme appliqu\'e \`a $P^0$. Ceci implique que $f_1$ est un isomorphisme et, au vu du diagramme pr\'ec\'edent, que $f_0$ est injective.

Supposons \`a pr\'esent qu'il existe un entier $m>0$ tel que $m\Phi=0$. 
La multiplication par $m$ est nulle sur le faisceau $i_*\Phi$, et constitue un automorphisme du faisceau $(\gmlog/\gm)^{\text{\rm kpl}}$ (d'apr\`es le lemme \ref{unidiv}). Il en r\'esulte que
$$
\ext^1_{\text{\rm kpl}}(i_*\Phi,(\gmlog/\gm)^{\text{\rm kpl}})=0
$$
On en d\'eduit, par passage \`a la cohomologie dans la suite \eqref{componentgroup}, que $f_2$ est injective. En appliquant le lemme du serpent au diagramme commutatif pr\'ec\'edent, on trouve que le conoyau de $f_0$ est nul, et $f_0$ est un isomorphisme.

La d\'emonstration du point $(ii)$ est essentiellement la m\^eme. On dispose en effet, d'apr\`es \cite[expos\'e VII, corollaire 3.6.5]{gro7} d'un isomorphisme canonique
$$
\biext^1_{\text{\rm kpl}}(P,Q;\gm)\simeq \ext^1_{\text{\rm kpl}}(P\otimes^{\mathbb{L}}Q;\gm)
$$
ce qui permet d'appliquer le m\^eme formalisme cohomologique que pr\'ec\'edemment.
\end{proof}

Le th\'eor\`eme \ref{sga7style} permet d'effectuer une relecture, sous un \'eclairage logarithmique, du r\'esultat bien connu de Grothendieck \cite[expos\'e VIII, th\'eor\`eme 7.1]{gro7} dans lequel sont construites des obstructions qui mesurent le d\'efaut de surjectivit\'e du morphisme de restriction $\ext^1_{\text{\rm pl}}(P,\gm)\longrightarrow \ext^1_{\text{\rm pl}}(P_{\eta},\gmeta)$ (ainsi que du morphisme analogue pour les groupes de biextensions).

Comme nous allons le voir, on peut retrouver \`a partir du th\'eor\`eme \ref{sga7style} les obstructions d\'efinies par Grothendieck.
Pour cela, nous introduisons la suite spectrale
\begin{equation}
\label{spectral1}
\begin{CD}
0\longrightarrow \ext^1_{\text{\rm pl}}(P,\gm)\longrightarrow \ext^1_{\text{\rm kpl}}(P,\gm) @>\delta >> \homr(P,R^1\varepsilon_* \gm) \\
\end{CD}
\end{equation}
dont le dernier terme n'est autre que
$\homr(P,i_*(\mathbb{Q}/\mathbb{Z})_k)=\homr(\Phi,(\mathbb{Q}/\mathbb{Z})_k)$,
ainsi que son alter ego dans le monde des biextensions
\begin{equation}
\label{spectral2}
\begin{CD}
0\longrightarrow \biext^1_{\text{\rm pl}}(P,Q;\gm)\longrightarrow \biext^1_{\text{\rm kpl}}(P,Q;\gm) @>\partial >> \biext^0(P,Q;R^1\varepsilon_* \gm) \\
\end{CD}
\end{equation}
dont le dernier terme est
$\homr(P\otimes Q,i_*(\mathbb{Q}/\mathbb{Z})_k)=\homr(\Phi\otimes\Psi,(\mathbb{Q}/\mathbb{Z})_k)$.

\begin{thm}
\label{retourmonodromie}
\begin{enumerate}
\item[$(i)$] Avec les notations du th\'eor\`eme \ref{sga7style} $(i)$, supposons qu'il existe un entier $m>0$ tel que $m\Phi=0$. Soit $E_{\eta}$ une extension de $P_{\eta}$ par $\gmeta$, et soit $E\in \ext^1_{\text{\rm kpl}}(P,\gm)$ l'unique prolongement de $E_{\eta}$. Alors le morphisme $d(E_{\eta})$ d\'efini par Grothendieck dans \cite[expos\'e VIII, 7.1.1]{gro7} est \'egal \`a l'image de $E$ par le morphisme $\delta$ de la suite \eqref{spectral1}.
\item[$(ii)$] Avec les notations du th\'eor\`eme \ref{sga7style} $(ii)$, supposons qu'il existe un entier $m>0$ tel que $m\Phi=0$ ou $m\Psi=0$. Soit $E_{\eta}$ une biextension de $(P_{\eta},Q_{\eta})$ par $\gmeta$, et soit $E\in \biext^1_{\text{\rm kpl}}(P,Q;\gm)$ l'unique prolongement de $E_{\eta}$. Alors le morphisme $d(E_{\eta})$ d\'efini par Grothendieck dans \cite[expos\'e VIII, 7.1.2]{gro7} est \'egal \`a l'image de $E$ par le morphisme $\partial$ de la suite \eqref{spectral2}.
\end{enumerate}
\end{thm}

\begin{proof}
Montrons le $(i)$. On peut r\'esumer la situation par le diagramme
\begin{equation}
\label{diagproof}
\begin{CD}
\ext^1_{\text{\rm kpl}}(P,\gm) @>\delta >> \homr(P,i_*(\mathbb{Q}/\mathbb{Z})_k)@. \;=\homr(\Phi,(\mathbb{Q}/\mathbb{Z})_k) \\
@V\tau VV @VVcV @. \\
\ext^1_{\text{\rm \'et}}(P,(\jcl)_*\gmeta) @>\lambda_*>> \ext^1_{\text{\rm \'et}}(P,i_*\mathbb{Z}_k)@. \\
\end{CD}
\end{equation}
dans lequel la fl\`eche $\tau$ est obtenue en combinant les isomorphismes
$$
\begin{CD}
\ext^1_{\text{\rm kpl}}(P,\gm) @>>> \ext^1_{\text{\rm pl}}(P_{\eta},\gmeta) @<<< \ext^1_{\text{\rm \'et}}(P,(\jcl)_*\gmeta) \\
\end{CD}
$$
du th\'eor\`eme \ref{sga7style} $(i)$ et du lemme \ref{ext1gmlog}. La fl\`eche $c$ est le cobord obtenu en appliquant le foncteur $\homr(P,-)$ \`a la suite exacte
$$
\begin{CD}
0 @>>> i_*\mathbb{Z}_k @>>> i_*\mathbb{Q}_k @>>> i_*(\mathbb{Q}/\mathbb{Z})_k @>>> 0 \\
\end{CD}
$$
lequel cobord est d'ailleurs un isomorphisme d'apr\`es \cite[expos\'e VIII, proposition 5.5 (ii)]{gro7}. La fl\`eche $\lambda_*$ est induite par le morphisme $\lambda$ de la suite \eqref{suite1bis}.

Soit $E\in \ext^1_{\text{\rm kpl}}(P,\gm)$, alors $\tau(E)$ est l'unique prolongement de $E_{\eta}$ en une extension de $P$ par $(\jcl)_*\gmeta$ sur le site lisse. D'autre part, il appara\^it clairement \`a la lecture de la d\'emonstration de \cite[expos\'e VIII, th\'eor\`eme 7.1]{gro7} que le morphisme $d(E_{\eta})$ est l'image de $\tau(E)$ par le morphisme $c^{-1}\circ \lambda_*$. Autrement dit, pour montrer le r\'esultat voulu, il suffit de montrer que \eqref{diagproof} est commutatif, ce que nous allons faire.

On constate tout d'abord, \`a la lueur du lemme \ref{ext1gmlog} et de la d\'emonstration du th\'eor\`eme \ref{sga7style}, que $\tau$ est obtenu en combinant les isomorphismes
$$
\begin{CD}
\ext^1_{\text{\rm kpl}}(P,\gm) @>>> \ext^1_{\text{\rm kpl}}(P,\gmlog) @<q^*<< \ext^1_{\text{\rm \'et}}(P,(\jcl)_*\gmeta) \\
\end{CD}
$$
ce qui se traduit concr\`etement de la fa\c{c}on suivante : soit $E\in \ext^1_{\text{\rm kpl}}(P,\gm)$, et soit $E'$ l'image de $E$ dans $\ext^1_{\text{\rm kpl}}(P,\gmlog)$. Par construction, nous avons un diagramme commutatif (\`a lignes exactes) de faisceaux pour la topologie log plate
$$
\begin{CD}
0 @>>> \gm @>>> E @>>> P @>>> 0 \\
@. @VVV @VVV @| \\
0 @>>> \gmlog @>>> E' @>>> P @>>> 0 \\
\end{CD}
$$
dont on d\'eduit, par le lemme des neuf, que $E'/E\simeq (\gmlog/\gm)^{\text{\rm kpl}}$.
En prenant l'image par le foncteur $q_*$ de ce diagramme sur le site lisse, on trouve, compte tenu de la nullit\'e de $R^1q_* \gmlog$, un diagramme commutatif \`a lignes exactes
$$
\begin{CD}
0 @>>> \gm @>>> q_*E @>>> P @>\delta(E) >> R^1q_* \gm \\
@. @VVV @VVV @| \\
0 @>>> (\jcl)_*\gmeta @>>> q_*E' @>>> P @>>> 0 \\
\end{CD}
$$
dans lequel la ligne du bas est l'extension $\tau(E)$. Mais on sait que $E'/E\simeq (\gmlog/\gm)^{\text{\rm kpl}}$ et par suite $q_*(E'/E)\simeq i_*\mathbb{Q}_k$ d'apr\`es le lemme \ref{fx7}. D'autre part, $R^1q_* \gm =i_*(\mathbb{Q}/\mathbb{Z})_k$ d'apr\`es le corollaire \ref{fx4}. En rassemblant tout cela on en d\'eduit le diagramme commutatif suivant, \`a lignes et colonnes exactes
$$
\begin{CD}
0 @>>> \gm @>>> q_*E @>>> \ker(\delta(E)) @>>> 0\\
@. @VVV @VVV @VVV \\
0 @>>> (\jcl)_*\gmeta @>>> q_*E' @>>> P @>>> 0 \\
@. @V\lambda VV @VVV @VV\delta(E) V \\
0 @>>> i_*\mathbb{Z}_k @>>>  i_*\mathbb{Q}_k @>>> i_*(\mathbb{Q}/\mathbb{Z})_k @>>> 0 \\
\end{CD}
$$
Examinons les deux lignes inf\'erieures de ce diagramme. Par un argument d'alg\`ebre homologique, on peut en d\'eduire que le pull-back de la suite du bas par le morphisme $\delta(E):P\rightarrow i_*(\mathbb{Q}/\mathbb{Z})_k$ est \'egal au push-forward de la suite du milieu (qui n'est autre que l'extension $\tau(E)$) par le morphisme $\lambda:(\jcl)_*\gmeta\rightarrow i_*\mathbb{Z}_k$. Autrement dit, $c(\delta(E))=\lambda_*(\tau(E))$, et par cons\'equent le diagramme \eqref{diagproof} est commutatif, ce qu'on voulait. La d\'emonstration du $(ii)$ est essentiellement la m\^eme.
\end{proof}

\begin{rmq}
Dans le cas o\`u $P^0$ est un sch\'ema semi-ab\'elien sur $S$ (c'est-\`a-dire que ses fibres sont extensions de vari\'et\'es ab\'eliennes pas des tores), il est probablement possible d'\'etablir un lien entre le th\'eor\`eme \ref{sga7style} et la th\'eorie des log sch\'emas ab\'eliens, d\'evelopp\'ee par Kajiwara, Kato et Nakayama dans \cite{kkn}.
\end{rmq}

\begin{rmq}
Avec les notations de \ref{sga7style} $(i)$, supposons qu'il existe un entier $m>0$ tel que $m\Phi=0$. Alors on dispose d'isomorphismes canoniques
$$
\ext^1_{\text{\rm kpl}}(P,\gm)\simeq \ext^1_{\text{\rm kpl}}(P,\gmlog) \simeq \ext^1_{\text{\rm pl}}(P_{\eta},\gmeta)
$$
Par contre, si l'on remplace la topologie Kummer plate par la topologie Kummer \'etale, alors ce r\'esultat n'est plus vrai en g\'en\'eral. En effet (cf. remarque \ref{logetalefailure}), la multiplication par $n$ dans le faisceau quotient $(\gmlog/\gm)^{\text{\rm k\'et}}$ n'est un isomorphisme que si $n$ est premier \`a la caract\'eristique r\'esiduelle $\ell$ du trait $S$. La partie de $\ell$-torsion du groupe $\Phi$ constitue alors une obstruction \`a prolonger les extensions en topologie Kummer \'etale.
\end{rmq}

On peut \'enoncer des r\'esultats analogues \`a ceux du th\'eor\`eme \ref{sga7style} dans le cas g\'en\'eral o\`u $S$ n'est plus forc\'ement un trait. On s'y ram\`ene en remarquant que, si $s$  est un point de codimension $1$ de $S$ qui appartient au support de $D$, alors la log structure induite sur le trait $\Spec(\mathcal{O}_{S,s})$ par celle de $S$ n'est autre que la log structure canonique.

\begin{cor}
\label{corutile}
Supposons que $S$ soit comme dans le paragraphe \ref{logref}. Si, pour chaque point $s$ de codimension $1$ de $S$ qui appartient au support de $D$, les hypoth\`eses du th\'eor\`eme \ref{sga7style} sont satisfaites, alors les conclusions analogues sont valables pour le morphisme $j^*$ de restriction \`a l'ouvert $U$.
\end{cor}


\subsection{Accouplements divers}
\label{accouplements}

A pr\'esent $S$ est (n\oe{}th\'erien r\'egulier) de dimension $1$, de point g\'en\'erique $\eta=\Spec(K)$.
Soit $\A$ (resp. $\A^t$) le mod\`ele de N\'eron d'une $K$-vari\'et\'e ab\'elienne (resp. de sa vari\'et\'e duale), on note $U\subseteq S$ l'ouvert de bonne r\'eduction de $\A$.
Nous munissons $S$ de la log structure d\'efinie par l'ouvert $U$.

On note $W_U$ la biextension de Weil de $(\A_U,\A^t_U)$ par $\gmu$ (pour la topologie fppf), laquelle traduit la dualit\'e entre les sch\'emas ab\'eliens $\A_U$ et $\A^t_U$. Son $\gmu$-torseur sous-jacent est le fibr\'e de Poincar\'e.

On note $\A^0$ le $S$-sch\'ema obtenu en recollant la fibre g\'en\'erique de $\A$ avec les composantes neutres de ses fibres sp\'eciales. On note $\Phi$ le faisceau quotient $\A/\A^0$  que l'on appelle, par abus de langage, le groupe des composantes de $\A$. En fait, on peut \'ecrire
$$
\Phi=\bigoplus_{s\in D} i_*\Phi_s=\bigoplus_{s\in D} i_*(\A_s/\A_s^0).
$$
On note $\Phi'$ le groupe des composantes de $\A^t$. On appelle (\'egalement par abus) accouplement de monodromie l'objet obtenu en recollant, pour chaque $s\in D$, l'accouplement  d\'efini par \cite[expos\'e IX, 1.2.1]{gro7} (apr\`es changement de base $\Spec(\mathcal{O}_{S,s})\rightarrow S$).

Si $\Gamma\subseteq\Phi$ est un sous-groupe ouvert de $\Phi$, on note $\A^{\Gamma}$ l'image r\'eciproque de $\Gamma$ par le morphisme naturel $\A\rightarrow \Phi$. Il s'agit d'un sous-groupe ouvert de $\A$.

Supposons que $\Gamma\subseteq\Phi$ et $\Gamma'\subseteq\Phi'$ soient des sous-groupes ouverts de $\Phi$ et $\Phi'$. Si $\Gamma$ et $\Gamma'$ sont orthogonaux sous l'accouplement de monodromie, alors il existe une unique biextension de $(\A^{\Gamma},\A^{t,\Gamma'})$ par $\gm$ prolongeant $W_U$, que nous noterons  $W^{\Gamma,\Gamma'}$.

\begin{prop}
\label{restrwlog}
\begin{enumerate}
\item[$(1)$] Il existe une unique biextension $W^{\rm log}$ de $(\A,\A^t)$ par $\gm$ (pour la topologie log plate) prolongeant la biextension de Weil $W_U$.
\item[$(2)$] Soit $\partial$ le morphisme
$$
\partial:\biext^1_{\text{\rm kpl}}(\A,\A^t;\gm) \longrightarrow \homr(\Phi\otimes\Phi',\bigoplus_{s\in D} i_*(\mathbb{Q}/\mathbb{Z})_s)
$$
d\'efini comme dans la suite exacte \eqref{spectral2}. Alors $\partial(W^{\rm log})$ n'est autre que l'accouplement de monodromie.
\item[$(3)$] Supposons que $\Gamma$ et $\Gamma'$ soient orthogonaux sous l'accouplement de monodromie. Alors la restriction de $W^{\rm log}$ \`a $(\A^{\Gamma},\A^{t,\Gamma'})$ est \'egale \`a $W^{\Gamma,\Gamma'}$.
\end{enumerate}
\end{prop}

\begin{proof}
$(1)$ D\'ecoule du corollaire \ref{corutile}. $(2)$ D\'ecoule du th\'eor\`eme \ref{retourmonodromie} $(ii)$.
$(3)$ Il suffit d'invoquer l'unicit\'e du prolongement de $W_U$ en une biextension de $(\A^{\Gamma},\A^{t,\Gamma'})$ par $\gm$ pour la topologie Kummer plate (th\'eor\`eme \ref{sga7style} $(ii)$).
\end{proof}

Si $E$ est une biextension de $(P,Q)$ par $\gm$, nous noterons $t(E)$ le $\gm$-torseur sous-jacent sur $P\times_S Q$.

Nous pouvons \`a pr\'esent d\'efinir deux accouplements. Le premier est celui construit par Mazur et Tate dans \cite[Remark 3.5.3]{mt}, le second est son analogue logarithmique.

\begin{dfn}
\label{defpairings}
\begin{enumerate}
\item[$(1)$] Supposons que $\Gamma$ et $\Gamma'$ soient orthogonaux sous l'accouplement de monodromie. L'accouplement de classes (relativement au couple $(\Gamma,\Gamma')$) est l'application bilin\'eaire d\'efinie par
\begin{equation*}
\begin{split}
\A^{\Gamma}(S)\times\A^{t,\Gamma'}(S)\;\longrightarrow & \;\pic(S) \\
(x,y)\;\longmapsto & \;(x\times y)^*(t(W^{\Gamma,\Gamma'})) \\
\end{split}
\end{equation*}
\item[$(2)$] L'accouplement de classes logarithmique est l'application bilin\'eaire d\'efinie par
\begin{equation*}
\begin{split}
\A(S)\times\A^t(S)\;\longrightarrow & \;H^1_{\text{\rm kpl}}(S, \gm) \\
(x,y)\;\longmapsto & \;(x\times y)^*(t(W^{\rm log})) \\
\end{split}
\end{equation*}
Nous noterons $<-,->^{\rm log}$ cet accouplement dans la suite.
\end{enumerate}
\end{dfn}

\begin{prop}
\label{proppairings}
\begin{enumerate}
\item[$(1)$] On dispose d'un diagramme commutatif
$$
\begin{CD}
\A^{\Gamma}(S)\times\A^{t,\Gamma'}(S) @>>> \pic(S) \\
@VVV @VVV \\
\A(S)\times\A^t(S) @>>> H^1_{\text{\rm kpl}}(S, \gm) \\
\end{CD}
$$
dans lequel les fl\`eches horizontales sont les accouplements, et les fl\`eches verticales sont les inclusions canoniques.
\item[$(2)$] Supposons que $S$ soit un trait. Alors on a un diagramme commutatif
$$
\begin{CD}
\A(S)\times\A^t(S) @>>> H^1_{\text{\rm kpl}}(S, \gm) \\
@VVV @| \\
\Phi(k)\times\Phi'(k) @>>> \mathbb{Q}/\mathbb{Z} \\
\end{CD}
$$
o\`u la fl\`eche horizontale du bas est induite par l'accouplement de monodromie, et la fl\`eche verticale de gauche est obtenue par r\'eduction \`a la fibre sp\'eciale.
\end{enumerate}
\end{prop}

\begin{proof}
Le point $(1)$ est une cons\'equence imm\'ediate de la proposition \ref{restrwlog}. Le point $(2)$ d\'ecoule du th\'eor\`eme \ref{retourmonodromie} $(ii)$, et de l'observation suivante (que l'on applique \`a la biextension $W^{\rm log}$) : \'etant donn\'e une biextension $E\in \biext^1_{\text{\rm kpl}}(P,Q;\gm)$, le morphisme $\partial(E):P\otimes Q\rightarrow R^1\varepsilon_* \gm$ de la suite spectrale \eqref{spectral2} s'obtient en faisceautisant (pour la topologie fppf, par rapport \`a l'argument $T$) le morphisme $P(T)\otimes Q(T)\rightarrow H^1_{\text{\rm kpl}}(T, \gm)$ qui \`a $x\otimes y$ associe $(x\times y)^*(t(E))$, $T$ d\'esignant ici un $S$-sch\'ema variable que l'on munit de la log structure image r\'eciproque de celle de $S$.
\end{proof}

En fait, il ressort de la lecture de \cite{mb} que l'accouplement de classes logarithmique peut s'exprimer en termes d'intersection. Ceci fait l'objet du paragraphe qui suit.


\subsection{Torseurs cubistes}

Nous renvoyons le lecteur \`a \cite[chap. I, 2.4.5]{mb} pour la notion de torseur cubiste dans un cadre toposique g\'en\'eral.
\'Etant donn\'es deux groupes commutatifs $A$ et $G$ d'un topos, la cat\'egorie des $G$-torseurs cubistes sur $A$, munie du produit tensoriel, est une cat\'egorie de Picard strictement commutative (au sens de \cite[expos\'e XVIII, 1.4.2]{gro4}). On note $\cub(A,G)$ le groupe des classes d'isomorphie d'objets de cette cat\'egorie.

Nous allons nous int\'eresser ici aux $\gm$-torseurs cubistes dans la cat\'egorie des faisceaux pour la topologie Kummer plate.

Nous reprenons les hypoth\`eses et notations du paragraphe pr\'ec\'edent.

\begin{rmq}
Comme le morphisme structural $f:\A\rightarrow S$ est lisse, la log structure induite sur $\A$ par celle de $S$ est la log structure associ\'ee \`a l'ouvert $f^{-1}(U)$. Ce dernier est le compl\'ementaire du diviseur $f^*(D)$ \'egal \`a la somme des fibres de $\A$ au-dessus des points de $D$. Les $\gm$-torseurs sur $\A$ pour la topologie Kummer plate sont donc d\'ecrits par le th\'eor\`eme \ref{logpic}.
\end{rmq}

\begin{prop}
\label{prolongecub}
Le morphisme de restriction
$$
\cub_{\text{\rm kpl}}(\A,\gm)\longrightarrow \cub_{\text{\rm pl}}(\A_U,\gmu)
$$
est un isomorphisme. Si $L$ est un $\gm$-torseur cubiste sur $\A_U$, nous noterons $L^{\rm log}$ son unique prolongement cubiste.
\end{prop}

\begin{proof}
En vertu de \cite[chap. I, 2.5.4]{mb}, on peut se ramener \`a un probl\`eme de prolongement de certaines biextensions de $(\A_U,\A_U)$ par $\gm$. L'existence et l'unicit\'e d'un tel prolongement d\'ecoulent alors du corollaire \ref{corutile}.
\end{proof}

La proposition ci-dessous est une simple retranscription de \cite[chap. III, 1.4]{mb}.

\begin{prop}
\begin{enumerate}
\item[$(i)$] Soit $L\in\pic(\A_U)$ un fibr\'e en droites sur $\A_U$. Soit $n$ un entier qui annule le groupe $\Phi$ des composantes de $\A$. Alors $L^{\otimes 2n}$ admet un unique prolongement cubiste \`a $\A$, on le note $(L^{\otimes 2n})^{\sim}$.

\item[$(ii)$] Soit $\sigma$ une section m\'eromorphe de $L$, et soit $(\sigma^{\otimes 2n})^{\sim}$ la section m\'eromorphe de $(L^{\otimes 2n})^{\sim}$ prolongeant $\sigma^{\otimes 2n}$. On d\'efinit un \'el\'ement $\diviseur(\sigma)^{\sim}$ de $\DivRat(\A,f^*(D))$ en posant
$$
\diviseur(\sigma)^{\sim}=\frac{1}{2n}\diviseur((\sigma^{\otimes 2n})^{\sim})
$$
On constate que $\diviseur(\sigma)^{\sim}$ est un prolongement de $\diviseur(\sigma)$, et satisfait le th\'eor\`eme du cube (\emph{i. e.} son $\mathcal{D}_3$ est principal, avec les notations de \cite[chap. I, 1.1]{mb}). En outre, il est caract\'eris\'e par ces propri\'et\'es, \`a un \'el\'ement de $\oplus_{s\in D}\, \mathbb{Z}.\A_s$ pr\`es.

\item[$(iii)$] Le prolongement cubiste logarithmique $L^{\rm log}$ de $L$ sur $\A$ est la classe de $\diviseur(\sigma)^{\sim}$ modulo $\Divp(\A)$.
\end{enumerate}
\end{prop}

\begin{proof}
$(i)$ Voir \cite[chap. II, 1.2.1]{mb} dans le cas o\`u $S$ est un trait. Le cas o\`u $S$ est r\'egulier de dimension $1$ en d\'ecoule ais\'ement. $(ii)$ Voir \cite[chap. III, 1.4]{mb} dans le cas o\`u $S$ est un trait. $(iii)$ D'apr\`es le point pr\'ec\'edent, la classe de $\diviseur(\sigma)^{\sim}$ est un $\gm$-torseur (pour la topologie Kummer log plate) cubiste qui prolonge $L$. Il est donc \'egal \`a $L^{\rm log}$ par unicit\'e d'un tel prolongement (prop. \ref{prolongecub}).
\end{proof}

On peut \`a pr\'esent traduire l'accouplement de classes en termes d'intersection.

\begin{prop}
\label{intersection}
Soit $(x,y)\in \A(S)\times\A^t(S)$. Alors $y$ d\'efinit un $\gm$-torseur (cubiste) $L_y$ sur $\A_U$, lequel admet un unique prolongement cubiste $L_y^{\rm log}$ sur $\A$.
\begin{enumerate}
\item[$(a)$] Nous avons l'\'egalit\'e
$$
<x,y>^{\rm log}=x^*L_y^{\rm log}
$$
\item[$(b)$] Soit $\sigma$ une section m\'eromorphe de $L_y$ telle que $x$ ne rencontre pas $\diviseur(\sigma)$ sur la fibre g\'en\'erique, alors $L_y^{\rm log}$ est la classe du diviseur
$$
\diviseur(\sigma)^{\sim}=\overline{\diviseur(\sigma)}+\sum_{s\in D}\sum_{\gamma\in\Phi_{s}} m(\gamma).\gamma
$$
o\`u le premier terme de la somme est l'adh\'erence sch\'ematique de $\diviseur(\sigma)$ dans $\A$, et o\`u les $m(\gamma)$ appartiennent \`a $\frac{1}{2n}\mathbb{Z}$. L'accouplement de classes se calcule alors par intersection
$$
<x,y>^{\rm log}=[x^*(\overline{\diviseur(\sigma)})]+\sum_{s\in D}  [m(\gamma_s(x)).s]
$$
o\`u, pour chaque $s\in D$, $\gamma_s(x)$ est la composante de $\A_s$ dans laquelle $x$ se r\'eduit.
\item[$(c)$] Supposons que $S$ soit un trait, de valuation $v$. Avec les notations pr\'ec\'edentes, et modulo l'isomorphisme canonique $H^1_{\text{\rm kpl}}(S, \gm)\simeq \mathbb{Q}/\mathbb{Z}$, nous avons
$$
<x,y>^{\rm log}=<\diviseur(\sigma),x>_v ~({\rm mod}~\mathbb{Z})
$$
o\`u $<\diviseur(\sigma),x>_v$ est le symbole d\'efini dans \cite[chap. III, 1.3]{mb} (lequel est l'oppos\'e du symbole de N\'eron en vertu de \cite[chap. III, 1.3.1]{mb}).
\item[$(d)$] Avec les hypoth\`eses et notations de $(c)$, nous avons
$$
(x,y)_s^{\rm mono}=<\diviseur(\sigma),x>_v ~({\rm mod}~\mathbb{Z})
$$
o\`u $(x,y)^{\rm mono}_s$ est le r\'esultat de l'accouplement de monodromie entre les composantes $\gamma_s(x)$ et $\gamma_s(y)$ de $\A_s$ dans lesquelles $x$ et $y$ se r\'eduisent respectivement.
\end{enumerate}
\end{prop}

\begin{proof}
$(a)$ Le pull-back $({\rm id}_{\A}\times y)^*(W^{\rm log})$ est une extension de $\A$ par $\gm$ (pour la topologie log plate), par cons\'equent le $\gm$-torseur sous-jacent $({\rm id}_{\A}\times y)^*(t(W^{\rm log}))$ est un prolongement cubiste de $L_y$, donc est \'egal \`a $L_y^{\rm log}$ par unicit\'e. On peut d'autre part \'ecrire $<x,y>^{\rm log}=(x\times y)^*(t(W^{\rm log}))=x^*({\rm id}_{\A}\times y)^*(t(W^{\rm log}))$, d'o\`u le r\'esultat. Le point $(b)$ ne fait que traduire le premier en termes d'intersection. Le point $(c)$ est une cons\'equence du $(b)$ et de \cite[chap. III, 1.4]{mb}. Le point $(d)$ d\'ecoule du $(c)$ et de la proposition \ref{proppairings} $(2)$.
\end{proof}

\begin{rmq}
Nous retrouvons, dans le point $(c)$ de la proposition \ref{intersection}, l'expression de l'accouplement de monodromie en termes du symbole de N\'eron, 
due \`a Bosch et Lorenzini \cite[Theorem 4.4]{bl}. Il ressort d'autre part de la proposition \ref{intersection} $(d)$ que l'accouplement de monodromie et celui de Moret-Bailly sont bien \'egaux (le signe n\'egatif apparaissant dans \cite[Remark 4.5]{bl} nous semble erron\'e, le symbole utilis\'e par Moret-Bailly \'etant l'oppos\'e du symbole de N\'eron).
\end{rmq}

\begin{rmq}
Soit $(x,y)\in \A(S)\times\A^t(S)$ tel que $<x,y>^{\rm log}$ soit de $n$-torsion (ce qui est le cas en particulier si $x$ ou $y$ est de $n$-torsion). D'apr\`es la proposition \ref{proppairings} $(2)$, la fl\`eche $\nu$ du corollaire \ref{spectralpic} permet de retrouver l'accouplement de monodromie, plus exactement nous avons
$$
\nu(<x,y>^{\rm log})=\sum_{s\in D} (x,y)^{\rm mono}_s.s
$$
Par cons\'equent, avec les notations de la remarque \ref{nthpower}, tout $n$-rel\`evement du diviseur ci-dessus est une puissance $n$-i\`eme dans $\pic(S)$, puisque $<x,y>^{\rm log}$ provient d'un $\mu_n$-torseur pour la topologie log plate. L'auteur ignore si cette curiosit\'e \'etait d\'ej\`a connue.
\end{rmq}



\vskip 1cm

Jean Gillibert
\smallskip

Institut de Math\'ematiques de Bordeaux

Universit\'e Bordeaux 1

351, cours de la Lib\'eration

33405 Talence Cedex

France

\medskip

\texttt{jean.gillibert@math.u-bordeaux1.fr}

\end{document}